\documentclass[10pt,leqno]{amsart}
\usepackage{indentfirst, amssymb,amsmath,amsthm}
\usepackage{csquotes}
\newtheorem*{theoA}{Theorem A}
\newtheorem*{theoB}{Theorem B}
\newtheorem*{theoC}{Theorem C}
\newtheorem*{theoD}{Theorem D}
\newtheorem*{theoE}{Theorem E}
\newtheorem*{theoF}{Theorem F}
\newtheorem*{theoG}{Theorem G}

\newtheorem{theo}{Theorem}[section]
\newtheorem{lem}{Lemma}[section]
\newtheorem{cor}{Corollary}[section]

\newtheorem{exm}{Example}[section]
\newtheorem{defi}{Definition}[section]
\newtheorem{rem}{Remark}[section]
\newtheorem{ques}{Question}[section]
\newcommand{\ol}{\overline}
\newcommand{\be}{\begin{equation}}
\newcommand{\ee}{\end{equation}}
\newcommand{\beas}{\begin{eqnarray*}}
\newcommand{\eeas}{\end{eqnarray*}}
\newcommand{\bea}{\begin{eqnarray}}
\newcommand{\eea}{\end{eqnarray}}

\numberwithin{equation}{section}
\begin{document}
\title[On some sufficient conditions of the strong uniqueness polynomials]{On some sufficient conditions of the strong uniqueness polynomials}
\date{}
\author[A. Banerjee and B. Chakraborty ]{ Abhijit Banerjee  and Bikash Chakraborty }
\date{}
\address{ Department of Mathematics, University of Kalyani, West Bengal 741235, India.}
\email{abanerjee\_kal@yahoo.co.in, abanerjee\_kal@rediffmail.com
}
\email{bikashchakraborty.math@yahoo.com, bikashchakrabortyy@gmail.com}
\maketitle
\let\thefootnote\relax
\footnotetext{Type set by \AmS -\LaTeX}
\setcounter{footnote}{0}

\begin{abstract} In this paper we shall find some sufficient conditions for a uniqueness polynomial to be a strong uniqueness polynomial as this type of problem was never investigated by the researchers earlier. We also exhibit some examples to substantiate our theorems.
 \end{abstract}
\section{Introduction Definitions and Results}
In this paper we adopt the standard notations of the Nevanlinna theory of meromorphic functions as explained in (\cite{4.1}).\par
In the course of studying the factorization of meromorphic functions, Gross (\cite{4}) introduced the concept of a unique range set, which we define first.\par
A discrete set $S$ of $\mathbb{C}$ is called a Unique Range Set for meromorphic (entire) functions if there exists no pair of two distinct non-constant meromorphic (or entire) functions such that they have the same inverse images of $S$ counted with multiplicities.\par
Pertinent with the above definition during the last quarter century or so several authors presented many elegant results to enrich uniqueness theory. Actually a lot of efforts were being put on to find unique range sets which are different in nature as well as cardinalities (\cite{1.1},\cite{2}, \cite{2.1},\cite{5},\cite{6},\cite{8}).
\par The basic idea in studying unique range set is to construct a polynomial $P(z)$ with simple zeros whose zero set $S$ will be desired unique range set. In course of the development of this particular literature, viz-a-viz value distribution theory has judiciously been shifted towards finding that $\frac{P(f)}{P(g)}$ is a  non-zero constant.\par
This idea motivated researchers to define the definitions of so called Uniqueness Polynomial or Strong Uniqueness Polynomial.
\begin{defi}
 A polynomial $P$ in $\mathbb{C}$ is called a uniqueness polynomial for meromorphic (entire) functions, if for any two non-constant meromorphic (entire) functions $f$ and $g$, $P(f)\equiv P(g)$ implies $f\equiv g$.We say $P$ is a UPM (UPE) in brief.
\end{defi}
It is clear that any one degree polynomial is uniqueness polynomial but from (\cite{6}), no polynomial of degree $2$ or $3$ is a UPE. Thus to be a uniqueness polynomial for entire functions is of degree at least four. Now we demonstrate some uniqueness polynomials.
\begin{exm} (\cite{6}) Let $P(z)=z^{4}+a_{3}z^{3}+a_{2}z^{2}+a_{1}z+a_{0}$. Then $P$ is not a UPM. Also $P$ is a UPE if and only if $(\frac {a_{3}}{2})^{3}-\frac {a_{2}a_{3}}{2}+a_{1}\not=0$.\end{exm}
\begin{exm}\label{D}(\cite{6.1}) Let $P(z)=z^{n}+a_{n-1}z^{n-1}+...+a_{1}z+a_{0}~(n\geq4)$ be a monic polynomial. If there exist an integer $t$ with $1\leq t < n-2$ and $\gcd(n,t)=1$ such that $a_{n-1}=...=a_{t+1}=0$ but $a_{t} \neq 0$, then $P$ is a UPE.
\end{exm}
\begin{exm}\label{E}(\cite{6.1}) Let $P(z)=z^{n}+a_{m}z^{m}+a_{0}$ be a monic polynomial such that $\gcd(n,m)=1$ and  $a_{m} \neq 0$. If $n\geq 5$ and  $1\leq m < n-1$, then $P$ is a UPM.
\end{exm}
\begin{defi}
A polynomial $P$ in $\mathbb{C}$ is called a strong uniqueness polynomial
for meromorphic (entire) functions if for any non-constant meromorphic (entire) functions $f$ and $g$, $P(f)\equiv AP(g)$ implies $f\equiv g$, where $A$ is any nonzero constant. In this case we say $P$ is a SUPM (SUPE) in brief.
\end{defi}
It is clear from the above definitions that a SUPM(SUPE) is a UPM(UPE) but a UPM(UPE) may not be a SUPM(SUPE). However the following example shows that one degree polynomials are UPM(UPE) but may not be SUPM(SUPE).
\begin{exm}\label{ex1} Let $P(z)=az+b$ $(a\neq 0)$. Clearly $P(z)$ is a UPM(UPE) but for any non-constant meromorphic function(entire) $g$, if we take $f := cg-\frac{b}{a}(1-c)$  $(c \neq 0,1)$, then $P(f)=cP(g)$ but $f \neq g$.
\end{exm}
First we recall some existing strong uniqueness polynomials in the literature.\par
\begin{exm}(\cite{8})
The polynomial
$$P_{Y}(z)=z^{n}+az^{n-r}+b$$
is a uniqueness polynomial if $\gcd (n, r)=1$, $r\geq 2$, $ab\not=0$ and  $n\geq 6$.\par
Also from [p.79, Case 3, first part, (\cite{8})], it is clear that whenever $n\geq 2r+4$, $P_{Y}(f)\equiv cP_{Y}(g)$ where $c\not=0$ implies $P_{Y}(f)\equiv P_{Y}(g)$ and hence it is a strong uniqueness polynomial.
\end{exm}
Next we invoke the following polynomial introduced by Frank and Reinders in (\cite{2}).
\begin{exm}(\cite{2})
$$P_{FR}(z)=\frac{(n-1)(n-2)}{2}z^{n}-n(n-2)z^{n-1}+\frac{n(n-1)}{2}z^{n-2}-d~(d\not=0,1).$$
From (\cite{2}), we know $P_{FR}$ is a UPM if $n\geq 6$. Also from [p. 191, Case 2, (\cite{2})], it is clear that whenever $n\geq 8$, $P_{FR}(f)\equiv cP_{FR}(g)$ where $c\not=0$ implies $P_{FR}(f)\equiv P_{FR}(g)$, i.e., $P_{FR}$ is a strong uniqueness polynomial when $n\geq 8$.
\end{exm}
Very recently the first author (\cite{1}) introduced a new polynomial.
\begin{exm}(\cite{1})
$$P_{B}(z)=\sum\limits_{i=0}^{m} \binom{m}{i}\frac{(-1)^{i}}{n+m+1-i}z^{n+m+1-i} + c$$
where $c \neq 0, -\sum\limits_{i=0}^{m} \binom{m}{i}\frac{(-1)^{i}}{n+m+1-i}$.\par
The first author showed that $P_{B}$ is a uniqueness polynomial of degree 6 and strong uniqueness polynomial of degree 7 for $c=1$.
\end{exm}
As finding unique range set is the motivation of studying strong uniqueness polynomials, it is quite natural to assume that uniqueness polynomial has no multiple zero. In (\cite{2.1,2.2}), to find a necessary and sufficient condition for a monic polynomial without multiple zero to be a UPM, Fujimoto introduced a new definition which was recently been characterized in (\cite{1.2}) as critical injection property.\par
Let $P(z)$ be a monic polynomial without multiple zero whose derivatives has mutually distinct $k$ zeros given by $d_{1}, d_{2}, \ldots, d_{k}$ with multiplicities $q_{1}, q_{2}, \ldots, q_{k}$ respectively. Below we are demonstrating the definition of \textbf{critical injection property}.
\begin{defi}
A polynomial $P$ is said to satisfy critical injection property if $P(\alpha )\not =P(\beta )$ for any two distinct zeros $\alpha $, $\beta $ of the derivative $P'$.
\end{defi}
Clearly the meaning of critical injection property is that the polynomial $P$ is injective on the set of distinct zeros of $P^{'}$, which are known as critical points of $P$. Naturally a polynomial with this property may be called a \textbf{critically injective polynomial}. Thus a critically injective polynomial has
at-most one multiple zero.
\par The following theorem of Fujimoto  completely characterizes a monic polynomial with only simple zero to be a uniqueness polynomial.
 \par
\begin{theoA}(\cite{2.2}) Suppose that $P(z)$ is critically injective.  Then $P(z)$ will be a uniqueness polynomial if and only if $$\sum \limits_{1\leq l<m\leq k}q_{_{l}}q_{m}>\sum \limits_{l=1}^{k} q_{_{l}}.$$
In particular the above inequality is always satisfied whenever $k\geq 4$. When $k=3$ and $\max \{q_{1},q_{2},q_{3}\}\geq 2$ or when $k=2$, $\min \{q_{1},q_{2}\}\geq 2$ and $q_{1}+q_{2}\geq 5$, then also the above inequality holds.\end{theoA}
\begin{exm}\label{9}(\cite{2.1})
For $k=1$, taking $P(z)=(z-a)^{q}-b$ for some constants $a$ and $b$ with $b\not=0$ and an integer $q\geq 2$, it is easy to verify that for an arbitrary non-constant meromorphic function $g$ and a constant $c(\not=1)$ with $c^{q}=1$, the function $g:=cf+(1-c)a(\not=f)$ satisfies the condition $P(f)=P(g)$.
\end{exm}
\par Fujimoto also showed that the critical injection property of polynomial helps one to find a sufficient condition for a strong uniqueness polynomial. In this connection, Fujimoto proved the following theorem.
\begin{theoB}\label{h1}(\cite{2.1}) For a critically injective polynomial $P(z)$ with $k\geq4$, if $$P(d_{1})+P(d_{2})+\ldots+P(d_{k}) \neq 0,$$ then $P$ is a strong uniqueness polynomial.
\end{theoB}
\begin{rem}
As a application of Theorem B, Fujimoto himself proved that $P_{Y}(z)$ is a strong uniqueness polynomial for $n>r+1$ when $r\geq3$ and $\gcd (n, r)=1$ [see p. 1192, example 4.10., (\cite{2.1})] which is an improvement of a result of Yi in (\cite{8}).\par
\end{rem}
\begin{theoC}(\cite{2.2}) For a critically injective polynomial $P(z)$ with $k=3$, if $\max (q_{1},q_{2},q_{3})\geq 2$ and
$$\frac{P(d_{l})}{P(d_{m})}\neq \pm 1,~for ~1\leq l<m \leq 3,$$
$$\frac{P(d_{l})}{P(d_{m})}\neq \frac{P(d_{m})}{P(d_{n})}~for~any~permutation~(l,m,n)~of~(1,2.3),$$ then $P$ is a strong uniqueness polynomial.
\end{theoC}
\begin{theoD}(\cite{2.2}) For a critically injective polynomial $P(z)$ with $k=2$ and $q_{1}\leq q_{2}$, if
\begin{enumerate}
\item $q_{1}\geq 3$ and $P(d_{1})+P(d_{2})\not=0$ or
\item $q_{1}\geq 2$ and $q_{2}\geq q_{1}+3$,
\end{enumerate}
then $P$ is a strong uniqueness polynomial.
\end{theoD}
We noticed from the definitions and Example \ref{ex1} that uniqueness polynomials of degree one need not be strong uniqueness polynomials.
The following example shows that for higher degree polynomial also the same conclusion can be derived.
\begin{exm}\label{kly} Consider $P(z)=z^{n-r}(z^{r}+a)$ where $a$ is a non-zero complex number and $\gcd(n,r)=1$, $r\geq 2$ and $n\geq 5$.\par
Then $P$ is a uniqueness polynomial as shown in Example 1.3 but for any non-constant meromorphic function $g$ if we take $f=\omega g$  where $\omega$ is non-real $r-th$ root of unity. Then $P(f)=\omega^{n-r}P(g)$. Thus uniqueness polynomial may not be strong uniqueness polynomial.
\end{exm}
Though Fujimoto performed some remarkable investigations to find some sufficient conditions for a critically injective polynomial to be a strong uniqueness polynomial but so far no attempt have been made by any researchers to find some sufficient conditions for a UPM to be SUPM.
To deal in this new perspective is the main motivation of this paper.\\
\section{Main Results}
We have already seen from the Example \ref{9} that a polynomial having only one critical points can't be a uniqueness polynomial. So uniqueness polynomials has at least two critical points. Now we state our results.
\begin{theo}\label{th1} Suppose $P$ is a critically injective uniqueness polynomial of degree $n$ with simple zeros. Assume that $P$ has at least two critical points and among them let $\alpha$ and $\beta$ be the two critical points with maximum multiplicities.
Also assume that $z=\alpha$ is a $P(\alpha)$ point of $P(z)$ of order $p$ and $z=\beta$ is a $P(\beta)$ point of $P(z)$ of order $t$.
If $\max\{t,p\}+t+p\geq 5+n$ and $\{P(\alpha)+P(\beta)\} \not=0$, then $P(z)$ is a strong uniqueness polynomial.
\end{theo}
\begin{rem} As $\alpha$ and $\beta$ are critical points of $P$ so $t, p \geq 2$.
\end{rem}
\begin{exm} Consider the polynomial $$P_{FR}(z)=\frac{(n-1)(n-2)}{2}z^{n}-n(n-2)z^{n-1}+\frac{n(n-1)}{2}z^{n-2}-c,$$ where $n\geq6$ and $c \neq 0, 1,\frac{1}{2}$. 
\par
We see that 
 as $c \neq 0, 1$, $P_{FR}$ has only simple zeros .\par
Again as $c \neq \frac{1}{2}$, $P_{FR}(1)-P_{FR}(0)=1-2c \not=0$, it follows that $P$ is a critically injective polynomial.\par
Also $P_{FR}(z)-P_{FR}(1)=(z-1)^{3}R_{1}(z)$, where $R_{1}(z)$ has no multiple zero with $R_{1}(1)\not=0$ and  $P_{FR}(z)-P_{FR}(0)=z^{n-2}R_{2}(z)$,where $R_{2}(z)$ has no multiple zero with $R_{2}(0)\not=0$.\par Clearly, in view of Theorem A, $P_{FR}(z)$ is a uniqueness polynomial for $n\geq 5$.\par Thus using Theorem \ref{th1}, we get that $ P_{FR}(z)$ a SUPM if  $c \neq 0, 1,\frac{1}{2}$ and $\max\{n-2,3\}+(n-2)+3\geq 5+n$, i.e., $n\geq 6$.
\end{exm}
\begin{exm}\label{abn}  Consider the polynomial $$P_{B}(z)=\sum\limits_{i=0}^{m} \binom{m}{i}\frac{(-1)^{i}}{n+m+1-i}z^{n+m+1-i} + c,~c \neq0,-\lambda,-\frac{\lambda}{2},$$
where $\lambda=\sum\limits_{i=0}^{m} \binom{m}{i}\frac{(-1)^{i}}{n+m+1-i}$.\par
First we notice that in view of Lemma 2.2 of (\cite{1}) $\lambda \neq 0$.\par
Clearly $P'_{B}(z)=z^{n}(z-1)^{m}$, and as $c \neq 0, -\lambda$ $P_{B}$ has only simple zeros.\par
Again as $\lambda \neq 0$ we have $P_{B}(1) - P_{B}(0) \not=0$ and  hence $P_{B}$ is critically injective.
Also $P_{B}(z)-P_{B}(1)=(z-1)^{m+1}R_{3}(z)$, where $R_{3}(z)$ has no multiple zero with $R_{3}(1)\not=0$ and $P_{B}(z)-P_{B}(0)=z^{n+1}R_{4}(z)$, where $R_{4}(z)$ has no multiple zero with $R_{4}(0)\not=0$.\par
Now if $\min\{m,n\}\geq2$ and $m+n\geq5$, by Theorem A, $P_{B}(z)$ is a uniqueness polynomial.\par
Since $c \neq -\frac{\lambda}{2}$, $P_{B}(1) + P_{B}(0) \not=0$. So in view of Theorem \ref{th1}, $P_{B}(z)$ is a strong uniqueness polynomial if $\min\{m,n\}\geq2$ and $m+n\geq5$ and $\max\{m+1,n+1\}+(m+1)+(n+1)\geq 5+(m+n+1)$, i.e., $\max\{m,n\} \geq 3$. \end{exm}
\begin{rem} If we take $n=3,m=2$ or $n=2,m=3$ then by above discussion $P_{B}$ is a six degree strong uniqueness polynomial.
\end{rem}

Inspired by the above Example \ref{abn}, we first introduce most general form of $P_{B}(z)$ and we shall show that Theorem \ref{th1} is also applicable to it.
\begin{exm}\label{hii}
Let us define
$$P(z)=\sum\limits_{i=0}^{m} \sum\limits_{j=0}^{n} \binom{m}{i}\binom{n}{j}\frac{(-1)^{i+j}}{n+m+1-i-j}z^{n+m+1-i-j}a^{j}b^{i}+c=Q(z)+c,$$
where $a, b$ be two complex numbers such that $b\neq0$, $a\not=b$ and
$$c\not\in \{0,-Q(a),-Q(b),-\frac{Q(a)+Q(b)}{2}\}.$$

Clearly \beas P'(z) &=& \sum\limits_{i=0}^{m} \sum\limits_{j=0}^{n} \binom{m}{i}\binom{n}{j}(-1)^{i+j}z^{m+n-i-j}a^{j}b^{i}\\
&=& (\sum\limits_{i=0}^{m}(-1)^{i}\binom{m}{i}z^{m-i}b^{i})(\sum\limits_{j=0}^{n}\binom{n}{j}(-1)^{j}z^{n-j}a^{j})\\
&=& (z-b)^{m}(z-a)^{n}\eeas

Next we shall show that $P(z)$ is a critically injective polynomial. To this end first we note that
$P(z)-P(b)=(z-b)^{m+1}R_{5}(z)$ where $R_{5}(z)$ has no multiple zero and $R_{5}(b)\not=0$, and
 $P(z)-P(a)=(z-a)^{n+1}R_{6}(z)$ where $R_{6}(z)$ has no multiple zero and $R_{6}(a)\not=0$.\par
So $P(a)= P(b)$ implies $(z-b)^{m+1}R_{5}(z)=(z-a)^{n+1}R_{6}(z)$.
 As we choose $a \not=b$ so $R_{5}(z)$ has a factor $(z-a)^{n+1}$ which implies the polynomial $P$ is of degree at least $m+1+n+1$, a contradiction.\par
Not only that by the assumption on $c$ it is clear that $P(a)+P(b) \not=0$ and $P(a)P(b) \not=0$.\par Thus $P$ has no multiple zero.\par
Again by Theorem A, $P$ is a uniqueness polynomial if $m+n\geq 5$ and $\min\{m,n\}\geq 2$.\par
Thus if $m+n\geq 5$, $\max\{m,n\} \geq 3$ and $\min\{m,n\}\geq 2$ then $P$ is a strong uniqueness polynomial for meromorphic functions by Theorem \ref{th1}.
\end{exm}
\begin{rem} If we take $n=3, m=2$ or $n=2, m=3$ then by above discussion $P$ is a six degree strong uniqueness polynomial.
\end{rem}
\begin{cor}\label{hi}
If we take $a=0$ and $b=1$ in above example then we get Example \ref{abn}.
\end{cor}
\begin{rem}
If we take $a=0$ and $b\not=0$ in the previous example, then we have the following polynomial :
$$P(z)=\sum\limits_{i=0}^{m} \binom{m}{i}\frac{(-1)^{i}}{n+m+1-i}z^{n+m+1-i}b^{i} + c,$$
where $bc \neq 0$, $c\not=-b^{n+m+1}\lambda,\frac{-b^{n+m+1}\lambda}{2}$, where $\lambda$ is defined as in the previous example,
then clearly when $m+n\geq 5$, $\max\{m,n\}\geq 3$ and $\min\{m,n\}\geq 2$, $P$ is a strong uniqueness polynomial.
\end{rem}
\begin{rem}
The above examples are related to the strong uniqueness polynomials with two critical points. Now we are giving the following example where there are more than two critical points, and in view of Theorem 2.1, one can easily verify that it is a strong uniqueness polynomial.
\end{rem}
\begin{exm}
 Consider the polynomial $P(z)= z^{n}-\frac{n}{m}z^{m}+b$ where $n-m\geq 2$. Then it is clear that $P$ has at least three critical points.\par
 As $P'(z)=nz^{m-1}(z^{n-m}-1)$, so $P(z)-P(1)=(z-1)^{2}T_{1}(z)$, where $T_{1}(1)\not=0$ and $P(z)-P(0)=z^{m}T_{2}(z)$, where $T_{2}(0)\not=0$.\par
In view of Example 1.3 we have already seen that $P(z)$ is a uniqueness polynomial for $n-m\geq2$, $\gcd(m,n)=1$ and $n\geq5$.\par
 Thus
 applying Theorem \ref{th1}, $P$ is a strong uniqueness polynomial when $b \not\in\{0,\frac{n-m}{m},\frac{n-m}{2m}\}$ and $\max\{m,2\}+m-n\geq3$, $n-m\geq2$, $\gcd(m,n)=1$, $n\geq5$.
\begin{rem}
For $n=7, m=5$ with proper choice of $b$ we can have seven degree Hong-Xun Yi type strong uniqueness polynomial.
\end{rem}
\end{exm}
\begin{theo}\label{th41} Suppose $P$ is a critically injective uniqueness polynomial of degree $n$  with simple zeros having at least two critical points, say  $\gamma$ and $\delta$.
Assuming that the total number of $P(\gamma)$ and $P(\delta)$ points of $P$ are respectively $p$ and $q$ with  $|p-q|\geq 3$. If for any complex number $d \not\in\{P(\gamma),P(\delta)\}$, $(P(z)-d)$ has at least $\min\{p+3,q+3\}$ distinct zeros then $P(z)$ is a strong uniqueness polynomial.
\end{theo}
The following Corollary is an immediate consequence of the above Theorem.
\begin{cor}\label{cor2.1} Suppose $P$ is a critically injective uniqueness polynomial of degree $n$  with simple zeros having at least two critical points, say  $\gamma$ and $\delta$.
 Assuming the total number of $P(\gamma)$ and $P(\delta)$ points of $P$ are respectively $p$ and $q$.
If for any complex number $d \not\in \{P(\gamma),P(\delta)\}$, $(P(z)-d)$ has at least $q+3$ distinct zeros and $P(\delta)=1$, $(P(\gamma))^{2} \not=0,1$ then $P(z)$ is a strong uniqueness polynomial.
\end{cor}
\begin{exm}\label{h}
Consider the polynomial
$$P(z)=\sum\limits_{i=0}^{m} \binom{m}{i}\frac{(-1)^{i}}{n+m+1-i}z^{n+m+1-i}b^{i} + 1,$$
where we choose $b (\neq0)$ such a manner that $b^{n+m+1}\sum\limits_{i=0}^{m} \binom{m}{i}\frac{(-1)^{i}}{n+m+1-i}\not= -1,-2$.\par
We note that $P'(z)=(z-b)^{m}z^{n}$.
As $\min\{m,n\}\geq 2$, with the suitable choice of $b$, $P$ has no multiple zero. Also $P$ is critically injective.\par
If we take $m,n\in\mathbb{N}$ with $m+n\geq 5$ and $\min\{m,n\}\geq 2$ then by Theorem A, $P$ is a uniqueness polynomial. \par Clearly for any complex number $d \in \mathbb{C}\backslash\{P(0),P(b)\}$, $(P(z)-d)$ has exactly $m+n+1$ distinct zeros, otherwise there exist at least one complex number $\varsigma$ which is a zero of $(P(z)-d)$ of multiplicity at least 2. Consequently $P(\varsigma)=d$ and  $P'(\varsigma)=0$. That is, $d \in \{P(0),P(b)\}$, which is not possible.\par
So in view of Corollary \ref{cor2.1}, $P$ is a strong uniqueness polynomial if  $m+n\geq 5$, $\min\{m,n\}\geq 2$ and $n\geq 3$.
\end{exm}
\begin{rem}
The above Example gives the answer of the question raised in the paper (\cite{1}).
\end{rem}
We have observed from Example \ref{kly} that uniqueness polynomial may contain multiple zeros. However the two theorems so far stated are dealing with strong uniqueness polynomials with simple zeros. So natural question would be whether there exist a strong uniqueness polynomial which has multiple zeros? The next theorem shows that the answer is affirmative.\par
Next we shall demonstrate the following strong uniqueness polynomial with multiple zero of degree $n\geq 6$.
\begin{theo}\label{th2}
Let $$P(z)=z^{n}+az^{n-1}+bz^{n-2},$$ where $ab \neq0$ and $a^{2}=\lambda b$ where $\lambda =4(1-\frac{1}{(n-1)^{2}})$, then $P(z)$ is a strong uniqueness polynomial of degree $n\geq 6$.
\end{theo}
\begin{cor}
Let $$P(z)=z^{n}+az^{n-1}+bz^{n-2}+c,$$ where $ab \neq0$ and $a^{2}=\lambda b$ where $\lambda =4(1-\frac{1}{(n-1)^{2}})$, then $P(z)$ is a uniqueness polynomial of degree $n\geq 6$.
\end{cor}
\begin{rem}
It is easy to see that if $P(z)$ is strong uniqueness polynomial then for any non-zero constants $a,c$, $P(af+b)=cP(ag+b)$ gives $(af+b)=(ag+b)$, i.e, $P(az+b)$ is also strong uniqueness polynomial.
\end{rem}
\section{Lemmas}

\begin{lem} \label{s1} If $$\psi(t)=\lambda(t^{n-1}-A)^{2}-4(t^{n-2}-A)(t^{n}-A)$$ where $\lambda =4(1-\frac{1}{(n-1)^{2}})$ and $A\neq1,0$ then $\psi(t)=0$ has no multiple roots.
\end{lem}
\begin{proof}
Let $F(t)=\psi(e^{t})e^{(1-n)t}$ for $t\in \mathbb{C}$. Then by elementary calculations we get
 $$F(t)=(\lambda-4)(e^{(n-1)t}+A^{2}e^{-(n-1)t})+4A(e^{t}+e^{-t})-2A\lambda$$
 Clearly if $t=0$ then $\psi(t)\neq0$.\\
Now, if possible $\psi(z_{0})=\psi'(z_{0})=0$. As $z_{0}\neq0$ , there exist some $w_{0} \in \mathbb{C}$ such that $z_{0}=e^{w_{0}}$.\\
As $F'(t)=\psi'(e^{t})e^{(1-n)t}e^{t}-(n-1)\psi(e^{t})e^{(1-n)t}$ , so $F(w_{0})=F'(w_{0})=0$.\\
Thus $$(\lambda-4)(e^{(n-1)w_{0}}+A^{2}e^{-(n-1)w_{0}})=-4A(e^{w_{0}}+e^{-w_{0}})+2A\lambda$$
and
$$(\lambda-4)(e^{(n-1)w_{0}}-A^{2}e^{-(n-1)w_{0}})=-\frac{4A(e^{w_{0}}-e^{-w_{0}})}{n-1}$$
Therefore
\beas 4A^{2}(\lambda-4)^{2} &=& (\lambda-4)^{2}((e^{(n-1)w_{0}}+A^{2}e^{-(n-1)w_{0}})^{2}-(e^{(n-1)w_{0}}-A^{2}e^{-(n-1)w_{0}})^{2})\\
&=& (-4A(e^{w_{0}}+e^{-w_{0}})+2A\lambda)^{2}-(-\frac{4A(e^{w_{0}}-e^{-w_{0}})}{n-1})^{2}\\
&=& 4A^{2}\lambda^{2}-32A^{2}\lambda\cosh w_{0}+64A^{2}\cosh^{2} w_{0}-\frac{64A^{2}}{(n-1)^{2}}\sinh^{2}w_{0}
\eeas
i.e. $$(\cosh w_{0})^{2}\{16-\frac{16}{(n-1)^{2}}\}-8\lambda \cosh w_{0}+\{8\lambda-16+\frac{16}{(n-1)^{2}}\}=0$$
i.e., $(\cosh w_{0}-1)^{2}=0$ that is $\cosh w_{0}=1$ which implies $z_{0}+\frac{1}{z_{0}}=2$\\
Hence $z_{0}=1$ but $\psi(1)=(1-A)^{2}\neq0$ as $A\neq 1$. Thus our assumption is wrong.
\end{proof}
\begin{lem}(\cite{2}) \label{s1.1} If $$\psi(t)=\lambda(t^{n-1}-1)^{2}-4(t^{n-2}-1)(t^{n}-1)$$ where $\lambda =4(1-\frac{1}{(n-1)^{2}})$ then $\psi(1)=0$ with multiplicity four. All other zeros of $\psi(t)$ are simple.
\end{lem}
\begin{lem}\label{s2} If $$\psi(t)=\lambda(t^{n-1}-A)^{2}-4(t^{n-2}-A)(t^{n}-A)$$ where $\lambda =4(1-\frac{1}{(n-1)^{2}})$ and $A\neq1,0$ , $t \neq 1$ then $\psi(t)=0$ and $t^{n}-A=0$ has no common roots.
\end{lem}
\begin{proof}
If $\psi(t)=0$ and $t^{n}-A=0$ has a common root then by the expression of $\psi(t)$ we get $t^{n-1}-A=0$ and $t^{n}-A=0$.\\
So $A=t^{n}=tt^{n-1}=tA$, which is not possible as $A\neq0$ and $t\neq1$.
\end{proof}
\section{Proofs of the theorems}
\begin{proof} [\textbf{Proof of Theorem \ref{th1} }]
By the given conditions on $P$, we can write
 \begin{enumerate}
\item $P(z)-P(\alpha)= (z-\alpha)^{p}Q_{n-p}(z)$ where $Q_{n-p}(z)$ is a polynomial of degree $(n-p)$, $Q_{n-p}(\alpha)\not=0$ and
\item $P(z)-P(\beta)=(z-\beta)^{t}Q(z)$, where  $Q(z)$ is a polynomial of degree $(n-t)$ and $Q(\beta)\not=0$.
\end{enumerate}
As $\alpha,\beta$ are critical points of $P$, so $P(\alpha)\not=P(\beta)$ and $t,p\geq 2$. Hence $P(\alpha)P(\beta)\not=0$ as all zeros of $P$ are simple.\par
Now suppose for any two non-constant meromorphic functions $f$ and $g$ and a non-zero constant $A \in \mathbb{C}$, $$P(f) = AP(g).$$\\
Now we consider two cases :\\
\textbf{Case -1} $A \neq 1$.\\
From the assumption of the theorem, $P$ is satisfying $\max\{t,p\}+t+p\geq 5+n$ where $t$, $p$ are previously defined.\\
\textbf{Subcase -1.1}\par
First assume that $t\geq p$. Thus in this case we have $2t+p\geq 5+n$. We define $F=\frac{(f-\beta)^{t}Q(f)}{P(\beta)}$ and $G=\frac{(g-\beta)^{t}Q(g)}{P(\beta)}$. Thus
\bea\label{e1.1} F=AG+A-1.\eea
So by Mokhon'ko's Lemma (\cite{4.2}), we have $T(r,f)=T(r,g)+O(1)$.\\
\textbf{Subcase -1.1.1} $A \neq \frac{P(\alpha)}{P(\beta)}$.\\
Now by applying  the Second Fundamental Theorem we get,
\beas && 2nT(r,f)+O(1)=2T(r,F)\\
 &\leq& \overline{N}(r,F)+\overline{N}(r,0;F)+\overline{N}(r,\frac{P(\alpha)}{P(\beta)}-1;F)+\overline{N}(r,A-1;F)+S(r,F)\\
&\leq& \overline{N}(r,f)+\overline{N}(r,\beta;f)+(n-t) T(r,f)+\overline{N}(r,\alpha;f)+(n-p)T(r,f)+\\
&+&\overline{N}(r,0;G)+S(r,f)\\
&\leq& \overline{N}(r,f)+\overline{N}(r,\beta;f)+(n-t)T(r,f)+\overline{N}(r,\alpha;f)+(n-p)T(r,f)+\\
&+&\overline{N}(r,\beta;g)+(n-t) T(r,g)+S(r,f)\\
&\leq& (3n-2t-p+4)T(r,f)+S(r,f), \eeas
which is a contradiction as $2t+p\geq 5+n$.\\
\textbf{Subcase -1.1.2} $ A=\frac{P(\alpha)}{P(\beta)}$.\\
In this case $$P(\beta)F=P(\alpha)G+\{P(\alpha)-P(\beta)\}.$$\par
As $P(\alpha) \pm P(\beta) \not=0$ and $P(\alpha)P(\beta) \not=0$, so $\frac{P(\alpha)-P(\beta)}{P(\beta)}\not=-\frac{P(\alpha)-P(\beta)}{P(\alpha)}$.\\
Thus in view of the Second fundamental theorem we get that,
\beas && 2nT(r,g)+O(1)=2T(r,G)\\
&\leq& \overline{N}(r,G)+\overline{N}(r,0;G)+\overline{N}(r,-\frac{P(\alpha)-P(\beta)}{P(\alpha)};G)+\overline{N}(r,\frac{P(\alpha)-P(\beta)}{P(\beta)};G)+S(r,G)\\
&\leq& \overline{N}(r,g)+\overline{N}(r,\beta;g)+(n-t)T(r,g)+\overline{N}(r,0;F)+\overline{N}(r,\alpha;g)+(n-p)T(r,g)+\\
&+& S(r,g)\\
&\leq& (3+2n-t-p)T(r,g)+\overline{N}(r,\beta;f)+(n-t)T(r,f)+S(r,g)\\
&\leq& (3n-2t-p+4)T(r,g)+S(r,g), \eeas
which is a contradiction as $2t+p\geq 5+n$.\\
\textbf{Subcase -1.2}\\
Assume $t< p$. Thus in this case we have $t+2p\geq 5+n$.\par
We define $F=\frac{(f-\alpha)^{p}Q_{n-p}(f)}{P(\alpha)}$ and $G=\frac{(g-\alpha)^{p}Q_{n-p}(g)}{P(\alpha)}$. Next proceeding similarly to above we get contradiction if we interchange the place of $\alpha$ and $\beta$.\par
So we conclude that if a critically injective polynomial $P$ with no multiple zeros satisfy $\max\{t,p\}+t+p\geq 5+n$ then $P(f)=AP(g)$ always imply $A=1$.\\
\textbf{Case -2} $A=1$\\
Then as $P(z)$ is a UPM, we have $f\equiv g$.
\end{proof}
\begin{proof} [\textbf{Proof of Theorem \ref{th41}}]
By the given assumptions we may write
\begin{enumerate}
\item $P(z)-P(\gamma)= (z-\xi_{1})^{l_{1}}(z-\xi_{2})^{l_{2}}...(z-\xi_{p})^{l_{p}}$ with $\gamma=\xi_{1}$,
\item $P(z)-P(\delta)=(z-\eta_{1})^{m_{1}}(z-\eta_{2})^{m_{2}}...(z-\eta_{q})^{m_{q}}$ with $\delta=\eta_{1}$
\end{enumerate}
where $\xi_{i}\not=\xi_{j}$, $\xi_{i}\not=\eta_{j}$  and $\eta_{i}\not=\eta_{j}$ for all $i,j$.\par
As $P$ has no multiple zeros and $\gamma,\delta$ are critical points of $P$, so we have $P(\gamma)P(\delta) \not=0$. \\
Also $P(\gamma)\not=P(\delta)$ as $P$ is critically injective.\par
Suppose for any two non-constant meromorphic functions $f$ and $g$ and for any non-zero complex constant $A$, $$P(f) = AP(g).$$
Then by Mokhon'ko's Lemma (\cite{4.2}), \beas T(r,f)=T(r,g)+O(1)~and~S(r,f)=S(r,g).\eeas
Now we are considering two cases:\\
\textbf{Case-1} $A \neq 1$ and $A = P(\gamma)$.\\
Then $P(\gamma)\neq 1$ and
\be \label{e66e dana} P(f)-P(\gamma) = P(\gamma)(P(g)-1)\ee
\textbf{Subcase-1.1.} Suppose $P(\delta)\not=1$.\\
Let $\nu_k$ $(k=1,2,\ldots,l)$ be the distinct zeros of $(P(g)-1)$.\\
Then by the Second Fundamental Theorem, we get
\beas (l-1)T(r,g) &\leq& \overline{N}(r,\infty;g)+\sum\limits_{k=1}^{l}\overline{N}(r,\nu_{k};g)+S(r,g). \\
&\leq& T(r,g)+\sum\limits_{i=1}^{p}\overline{N}(r,\xi_{i};f)+S(r,g). \\
&\leq& (p+1)T(r,g)+S(r,g),\eeas
which is a contradiction.\\
\textbf{Subcase-1.2.} Next suppose $P(\delta)=1$.\\
Here we consider the following two cases.\\
\textbf{Subsubcase-1.2.1} Suppose $P(\gamma)=-1$. \\
Then $P(f)-P(\delta) = P(\gamma)(P(g)-P(\gamma))$.\\
In this case in view of the Second Fundamental Theorem we have
\beas (p-1)T(r,g) &\leq& \overline{N}(r,g)+\sum\limits_{i=1}^{p}\overline{N}(r,\xi_{i};g)+S(r,g) \\
&\leq& T(r,g)+\sum\limits_{j=1}^{q}\overline{N}(r,\eta_{j};f)+S(r,g) \\
&\leq& (q+1)T(r,g)+S(r,g),\eeas
which leads to a contradiction.\\
\textbf{Subsubcase-1.2.2} Let $P(\gamma)\not=-1$.\\
Thus $P(f)-P(\delta) = P(\gamma)(P(g)-\frac{P(\delta)}{P(\gamma)})$, where $\frac{P(\delta)}{P(\gamma)} \not\in \{1,P(\delta),P(\gamma)\}$.\\
Let $\theta_{l}$ $(l=1,2,...,t)$ be the distinct zeros of $(P(z)-\frac{P(\delta)}{P(\gamma)})$.\\
Then by the Second Fundamental Theorem, we get
\beas (t-2)T(r,g) &\leq& \sum\limits_{l=1}^{t}\overline{N}(r,\theta_{l};g)+S(r,g) \\
&\leq& \sum\limits_{j=1}^{q}\overline{N}(r,\eta_{j};f)+S(r,f) \\
&\leq& q T(r,g)+S(r,g),\eeas
which is a again contradiction.\\
\textbf{Case-2.} $A \neq 1$ and $A \neq P(\gamma)$.\\
In this case, $P(f)-AP(\delta) = A(P(g)-P(\delta))$.\\
\textbf{Subcase-2.1.} $AP(\delta)\neq P(\gamma)$.\\
Thus $AP(\delta)\not\in \{P(\gamma),P(\delta)\}$.\\
Let $\zeta_{k}$ $(k=1,2,\ldots,m)$ be the distinct zeros of $(P(z)-AP(\delta))$.\\
Then by the Second Fundamental Theorem we get
\beas (m-2)T(r,f) &<& \sum\limits_{k=1}^{m}\overline{N}(r,\zeta_{k};f)+S(r,f) \\
&\leq& \sum\limits_{j=1}^{q}\overline{N}(r,\eta_{j};g)+S(r,g) \\
&\leq& q T(r,g)+S(r,g),\eeas
which is not possible.\\
\textbf{Subcase-2.2.} $AP(\delta)= P(\gamma)$.\\
Thus $P(\delta) \not=1$ and $P(f)-P(\gamma) = A(P(g)-P(\delta))$.\\
By the Second Fundamental Theorem, we get
\beas (p-2)T(r,f) &<& \sum\limits_{i=1}^{p}\overline{N}(r,\xi_{i};f)+S(r,f) \\
&\leq& \sum\limits_{j=1}^{q}\overline{N}(r,\eta_{j};g)+S(r,g) \\
&\leq& q T(r,g)+S(r,g),\eeas
Proceeding similarly we get
$$(q-2)T(r,g)\leq p T(r,f)+S(r,f).$$
Since $|p-q|\geq 3$, in either cases we get a contradiction.\\
Thus $A=1$. Hence as $P$ is a UPM, we get $f\equiv g$.
\end{proof}
\begin{proof} [\textbf{Proof of Theorem \ref{th2} }]
Suppose  $f$ and $g$ be two non-constant meromorphic functions such that $P(g)=AP(f)$ where $A\in\mathbb{C}\setminus\{0\}$. Then by Mokhon'ko's Lemma (\cite{4.2}), \beas T(r,f)=T(r,g)+O(1)~and~S(r,f)=S(r,g).\eeas
Putting $h=\frac{f}{g}$, we get $$g^{2}(h^{n}-A)+ag(h^{n-1}-A)+b(h^{n-2}-A)=0.$$
If $h$ is a constant function then as $g$ is non-constant, we get $(h^{n}-A)=(h^{n-1}-A)=(h^{n-2}-A)=0$.\\
i.e., $A=Ah=Ah^{2}$ which gives $h=1$ and hence $f=g$.\par
Next we consider $h$ as non-constant.\\
Then \bea \label{decem} \left(g+\frac{a}{2}\frac{h^{n-1}-A}{h^{n}-A}\right)^{2}&=& \frac{b\psi(h)}{4(h^{n}-A)^{2}},\eea
where $\psi(t)=\lambda(t^{n-1}-A)^{2}-4(t^{n-2}-A)(t^{n}-A)$.\\
\textbf{Case-1.} $A=1$.\par
Clearly in view of Lemmas \ref {s1.1} and \ref {s2} from the equation (\ref{decem}) have $$\left(g+\frac{a}{2}\frac{h^{n-1}-1}{h^{n}-1}\right)^{2}=\frac{b(h-1)^{4}\prod\limits_{i=1}^{2n-6}(h-\kappa_{i})}{4\{(h-1)\prod\limits_{j=1}^{n-1}(h-\rho_{j})\}^{2}},$$
where $\kappa_{i}\not =\rho_{j}$ for $i=1,\ldots,2n-6;j=1,\ldots,n-1$.
\\
Now by the Second Fundamental Theorem, we get
\beas (3n-9)T(r,h) &\leq& \sum\limits_{i=1}^{2n-6}\ol{N}(r,\kappa_{i};h)+\sum\limits_{j=1}^{n-1}\ol{N}(r,\rho_{j};h)+S(r,h)\\
&\leq& \frac{1}{2}\sum\limits_{i=1}^{2n-6}N(r,\kappa_{i};h)+\sum\limits_{j=1}^{n-1}\ol{N}(r,\rho_{j};h)+S(r,h)\\
&\leq& (2n-4)T(r,h)+S(r,h) \eeas
which is a contradiction for $n\geq 6$.\\
\textbf{Case-2.} $A\neq1$.\\
From (\ref{decem}) we have,
 $$\left(g+\frac{a}{2}\frac{h^{n-1}-A}{h^{n}-A}\right)^{2}=\frac{b\psi(h)}{4(h^{n}-A)^{2}}.$$
By Lemma \ref{s1} $\psi(t)=0$ has $(2n-2)$ distinct zeros, say $\zeta_{i}$ for $i=1,\ldots,2n-2$. So in view of Lemma \ref{s2} and the Second Fundamental Theorem we get
\beas (2n-4)T(r,h) &\leq& \sum\limits_{i=1}^{2n-2}\ol{N}(r,\zeta_{i};h)+S(r,h)\\
&\leq& \frac{1}{2}\sum\limits_{i=1}^{2n-2}N(r,\zeta_{i};h)+S(r,h)\\
&\leq& (n-1)T(r,h)+S(r,h) \eeas
which is a contradiction when $n\geq 4$.\\
\end{proof}
\section{Applications}
We observe from the discussion that at the time of studying uniqueness polynomial, it is general curiosity of the researchers to investigate whether the zero set of the uniqueness polynomial forms unique range set or not. For example Yi (\cite{8}), Frank-Reiders (\cite{2}), the present first author (\cite{1}) simultaneously studied the corresponding unique range sets in connection to their uniqueness polynomial.
Though the motivation of this paper is to give some sufficient conditions for strong uniqueness polynomials and simultaneously reduce the degree of some existing strong uniqueness polynomials,
but as we have already introduced some new type of uniqueness polynomials in Example \ref{hii}, we also intend to follow the same direction. In other words below we demonstrate new type of unique range sets by taking the zero sets of the strong uniqueness polynomials in the Example \ref{hii}.\par
Before going to state our concerning result we recall some well known definitions and results.\\
Let $f$ and $g$ be two non-constant meromorphic functions and let $a$ be a finite complex number. We say that $f$ and $g$ share the value $a-$CM (counting multiplicities), provided that $f-a$ and $g-a$ have the same zeros with the same multiplicities. Similarly, we say that $f$ and $g$ share the value $a-$IM (ignoring multiplicities), provided that $f-a$ and $g-a$ have the same set of zeros, where the multiplicities are not taken into account.
In addition we say that $f$ and $g$ share $\infty$ CM (IM), if $1/f$ and $1/g$ share $0$ CM (IM).

Let $S$ be a set of distinct elements of $\mathbb{C}\cup\{\infty\}$ and $E_{f}(S)=\bigcup_{a\in S}\{z: f(z)=a\}$, where each zero is counted according to its multiplicity. If we do not count the multiplicity, then the set $ \bigcup_{a\in S}\{z: f(z)=a\}$ is denoted by $\ol E_{f}(S)$.
If $E_{f}(S)=E_{g}(S)$ we say that $f$ and $g$ share the set $S$ CM. On the other hand, if $\ol E_{f}(S)=\ol E_{g}(S)$, we say that $f$ and $g$ share the set $S$ IM.\par
Let a set $S\subset \mathbb{C} $ and $f$ and $g$ be two non-constant meromorphic (entire) functions. If $E_{f}(S)=E_{g}(S)$ implies $f\equiv g$ then $S$ is called a unique range set for meromorphic (entire) functions, in short URSM (URSE). \par
The analogous definition for reduced unique range sets are as follows :\\
 We shall call any set $S\subset \mathbb{C} $ a unique range set for meromorphic (entire) functions ignoring multiplicity (URSM-IM) (URSE-IM) or a reduced unique range set for meromorphic (entire) functions (RURSM) (RURSE) if $\ol E_{f}(S)=\ol E_{g}(S)$ implies $f\equiv g$ for any pair of non-constant meromorphic (entire) functions.\par
 Fujimoto first showed that the critical injection property of polynomials helps one to find sufficient condition for a set of zeros $S$ of a SUPM(SUPE) $P$ to be a URSM (URSE).
 \begin{theoE}\label{1111}(\cite{2.1})
 Let $P$ be a critically injective polynomial of degree $n$ in $\mathbb{C}$ having only simple zeros. Let $P'$ have $k$ distinct zeros and either $k\geq3$ or $k=2$ and $P'$ have no simple zero. Further suppose that $P$ is a SUPM(SUPE). If $S$ is the set of zeros of $P$, then $S$ is a URSM(URSE) whenever $n>2k+6(n>2k+2)$ while URSM-IM(URSE-IM) whenever $n>2k+12(n>2k+5)$.
 \end{theoE}
Next we recall another definition.
\begin{defi}(\cite{1})
A set $S\subset \mathbb{C} $ is called a $URSM_{l)}$ ($URSE_{l)}$) if for any two non-constant meromorphic (entire) functions $f$ and $g$, $E_{l)}(S,f)=E_{l)}(S,g)$ implies $f\equiv g$ . \par
\end{defi}
In 2009 Bai, Han and Chen (\cite{1.0}) improved Theorem E as follows.
\begin{theoF}(\cite{1.0})
In addition to the hypothesis of Theorem E we suppose that $l$ is a positive integer or $\infty$. Let $S$ be the set of zeros of $P$. If
\begin{enumerate}
\item $l\geq 3$ or $\infty$ and  $n>2k+6(n>2k+2)$,
\item $l=2$  and  $n>2k+7(n>2k+2)$,
\item $l=1$  and  $n>2k+10(n>2k+4)$,
\end{enumerate}
then $S$ is a $URSM_{l)}$ ($URSE_{l)}$.
\end{theoF}
Recently the present first author proved the following result (\cite{1}) in more general settings.
\begin{theoG}(\cite{1})
In addition to the hypothesis of Theorem E we suppose that $l$ is a positive integer or $\infty$. Let $S$ be the set of zeros of $P$. If
\begin{enumerate}
\item $l\geq 3$ or $\infty$ and  $\min\{\Theta(\infty;f),\Theta(\infty;g)\}>\frac{6+2k-n}{4}$,
\item $l=2$  and  $\min\{\Theta(\infty;f),\Theta(\infty;g)\}>\frac{14+4k-2n}{9}$,
\item $l=1$  and  $\min\{\Theta(\infty;f),\Theta(\infty;g)\}>\frac{10+2k-n}{6}$,
\end{enumerate}
then $S$ is a $URSM_{l)}$ ($URSE_{l)}$.
\end{theoG}
We have already seen from Example \ref{hii} that the polynomial
\begin{equation}\label{equ1} P(z)=\sum\limits_{i=0}^{m} \sum\limits_{j=0}^{n} \binom{m}{i}\binom{n}{j}\frac{(-1)^{i+j}}{n+m+1-i-j}z^{n+m+1-i-j}a^{j}b^{i} + c,\end{equation}
is a critically injective strong uniqueness polynomial without any multiple zeros when $m+n\geq 5$, $\max\{m,n\}\geq 3$ and $\min\{m,n\}\geq 2$ with $a\not= b$, $b\not =0$.
Also we have defined
$$Q(z)=\sum\limits_{i=0}^{m} \sum\limits_{j=0}^{n} \binom{m}{i}\binom{n}{j}\frac{(-1)^{i+j}}{n+m+1-i-j}z^{n+m+1-i-j}a^{j}b^{i}$$
choose
$$c\not\in \{0,-Q(a),-Q(b),-\frac{Q(a)+Q(b)}{2}\}.$$
\\
So by the help of Theorem F and G, the following two theorems are obvious.
\begin{theo}\label{bcha} Let $m,n$ be two integers such that  $m+n\geq 5$, $\max\{m,n\}\geq 3$ and $\min\{m,n\}\geq 2$. Take $S=\{z~:~P(z)=0\}$ where $P$ is defined by \ref{equ1} with the already defined choice of $a,b,c$.\par
We suppose that $l$ is a positive integer or $\infty$. If
\begin{enumerate}
\item $l\geq 3$ or $\infty$ and  $m+n>9(5)$,
\item $l=2$  and  $m+n>10(5)$,
\item $l=1$  and  $m+n>13(7)$,
\end{enumerate}
then $S$ is a $URSM_{l)}$ ($URSE_{l)}$.
\end{theo}
\begin{theo} With the suppositions of Theorem \ref{bcha}, if
\begin{enumerate}
\item $l\geq 3$ or $\infty$ and  $\min\{\Theta(\infty;f),\Theta(\infty;g)\}>\frac{9-m-n}{4}$,
\item $l=2$  and  $\min\{\Theta(\infty;f),\Theta(\infty;g)\}>\frac{20-2m-2n}{9}$,
\item $l=1$  and  $\min\{\Theta(\infty;f),\Theta(\infty;g)\}>\frac{13-m-n}{6}$,
\end{enumerate}
then $S$ is a $URSM_{l)}$ ($URSE_{l)}$.
\end{theo}
\begin{cor}
consider the polynomial
$$P(z)=\sum\limits_{i=0}^{m} \binom{m}{i}\frac{(-1)^{i}}{n+m+1-i}z^{n+m+1-i}b^{i} + c,$$
where $bc \neq 0$ , $c\not=-b^{n+m+1}\lambda,-\frac{b^{n+m+1}\lambda}{2}$ where $\lambda=\sum\limits_{i=0}^{m} \binom{m}{i}\frac{(-1)^{i}}{n+m+1-i}$ and $m+n\geq 5$, $\max\{m,n\}\geq 3$, $\min\{m,n\}\geq 2$.\par
Clearly Lemma 2.2 of (\cite{1}) implies $P(b)-P(0)=b^{n+m+1}\lambda\not=0$, which implies $P$ is critically injective. Again as $P(0)=c\not=0$ and $P(b)\not=0$,
it follows that $P$ have no multiple zeros.\par
Finally, as $P(b)+P(0)=b^{n+m+1}\lambda+2c\not=0$ and  $m+n\geq 5$, $\max\{m,n\}\geq 3$ , $\min\{m,n\}\geq 2$, by Theorem \ref{th1}, $P$ is a strong uniqueness polynomial.
\par
Thus the last two theorems are applicable for this polynomial.
\end{cor}
\section{Concluding Remarks}
We see that Theorem \ref{th2} assures the existence of a uniqueness polynomial with multiple zero which is a strong uniqueness polynomial. On the other hand, Example \ref{kly} exhibits a uniqueness polynomial with multiple zero which is not a strong uniqueness polynomial.
Also Example \ref{hii} shows that a uniqueness polynomial with simple zeros may be a strong uniqueness polynomial. Thus following question is inevitable.
\begin{ques}
Whether there exist any uniqueness polynomial of degree $\geq2$ with simple zeros which is not a strong uniqueness polynomial?
\end{ques}
\begin{center} {\bf Acknowledgement} \end{center}
 The first author's research work is supported by the Council Of  Scientific and Industrial Research, Extramural Research Division, CSIR Complex, Pusa, New Delhi-110012, India, under the sanction project no. 25(0229)/14/EMR-II. \par
  The second author's research work is supported by the Department of Science and Technology, Govt. of India under the sanction order DST/INSPIRE Fellowship/2014/IF140903.


\end{document}